\documentclass[bibtex,12pt,en]{elegantpaper}

\usepackage{extarrows}
\usepackage{esint}
\usepackage{mathrsfs}

\numberwithin{equation}{section}
\allowdisplaybreaks[4]
\everymath{\displaystyle}

\newcommand{\R}{\mathbb R}

\newcommand{\dd}{\,dx}
\DeclareMathOperator{\diver}{div}

\title{The $m$-Laplace equation with a gradient term I: Liouville theorem in the second subcritical case}

\author{Tian Wu \and Jin Yan \and Hua Zhu}

\date{}

\begin{document}

\maketitle

\renewcommand{\thefootnote}{\fnsymbol{footnote}}


\begin{abstract}
\hspace{2em}Let $1<m<n$, $q>0$, and $p\in\mathbb R$. We study positive
$C^1_{\mathrm{loc}}$ weak solutions of
\[
-\Delta_m u=u^p|Du|^q
\qquad\text{in }\mathbb R^n.
\]
If $0<q<m-1$, we prove that every such solution is constant
provided
\[
p+q-m+1<
\frac{(m-1)(m-q)^2}
{(n-m)(m-1-q)}.
\]
If $q\geqslant m-1$, the same conclusion holds for every
$p\in\mathbb R$.

For $0<q<m$, the proof uses an auxiliary function and reduces the
main differential estimate to a lower bound for a function of one
variable. An explicit construction and a local maximum-principle
argument then prove constancy without differentiating across
$\{|Du|=0\}$. When $m=2$, our result answers the natural Liouville problem left open by Bidaut-V\'eron, Garc\'ia-Huidobro and V\'eron
for the whole range $0<q<1$. The case $q\geqslant m$ is handled by the related work of Bidaut-V\'eron.

\keywords{$m$-Laplace equation, product gradient term,
Liouville-type theorem}\\
\textbf{2020 Mathematics Subject Classification:}
Primary 35B53; Secondary 35J92, 35B33.
\end{abstract}


\section{Introduction}

Liouville-type theorems are a classical topic in mathematical analysis. They describe parameter ranges in which an equation admits no nonconstant entire solution of a prescribed sign or regularity. Such results are fundamental in the study of a priori estimates, blow-up analysis, existence theory, and the classification of solutions.

For the semilinear Lane--Emden equation, Gidas and Spruck \cite{GS1981} established the classical Liouville theorem in the Sobolev-subcritical range. Serrin and Zou \cite{SerrinZou} developed a general Cauchy--Liouville theory and universal estimates for quasilinear equations of the form
\begin{align}\label{m-eq}
 \Delta_m u+f(u)=0\quad\text{in }\R^n.
\end{align}
Their results include the power case and extend the classical semilinear theory to the $m$-Laplacian.

Motivated by the work of Bidaut-V\'eron, Garc\'ia-Huidobro, and V\'eron \cite{BV-GH-V2019}, we study equations in which the reaction term is the product of a power of the solution and a power of its gradient. More precisely, we consider
\begin{equation}\label{eq:pde}
 -\Delta_m u=u^p|Du|^q\quad\text{in }\R^n,
 \qquad
 \Delta_m u=\diver\bigl(|Du|^{m-2}Du\bigr),
\end{equation}
where $1<m<n$, $q>0$, and $p\in\R$.
A positive $C^1_{\mathrm{loc}}$ weak solution is a function
$u\in C^1_{\mathrm{loc}}(\R^n)$, with $u>0$, such that
\begin{equation*}\label{eq:weak}
 \int_{\R^n}|Du|^{m-2}Du\cdot D\varphi\dd
 =\int_{\R^n}u^p|Du|^q\varphi\dd,
 \qquad \varphi\in C_c^\infty(\R^n).
\end{equation*}
Both integrands are locally integrable: on every compact set, $u$ is bounded away from zero and $u$ and $Du$ are bounded.

\begin{theorem}\label{thm:main}
Let $u$ be a positive $C^1_{\mathrm{loc}}$ weak solution of \eqref{eq:pde}.
Then $u$ is constant in either of the following ranges:
\begin{align}
 &0<q<m-1,\qquad
 p+q-m+1<\frac{(m-1)(m-q)^2}{(n-m)(m-1-q)};
 \label{eq:low}\\[2pt]
 &m-1\leqslant q<m,\qquad p\in\R.
 \label{eq:high}
\end{align}
\end{theorem}

Theorem~\ref{thm:main} is the main result of this paper. The range $q\geqslant m$ was studied by Bidaut-V\'eron \cite{BidautVeron2021}, who established Liouville theorems for $p\geqslant 0$ and, more generally, for $p\leqslant 0$ under the condtion $p+q+1-m>0$. In the borderline case $q=m$, an increasing change of the dependent variable reduces the equation to the $m$-harmonic equation. This trancformation was also used in \cite{BidautVeron2021}, and we include a short proof covering every $p\in\mathbb R$ for completeness. For $q>m$, the more recent theorem of Lu and Zhu \cite{LuZhu} extends the Liouville property to all $p\in\mathbb R$. Together with Theorem \ref{thm:main}, these results yield the following corollary.

\begin{corollary}\label{cor:complete}
Let $u$ be a positive $C^1_{\mathrm{loc}}$ weak solution of
\eqref{eq:pde}. Then $u$ is constant in either of the following
ranges:
\begin{align*}
 &0<q<m-1,\qquad
 p+q-m+1<
 \frac{(m-1)(m-q)^2}{(n-m)(m-1-q)};\\[2pt]
 &q\geqslant m-1,\qquad p\in\R.
\end{align*}
\end{corollary}

There is no lower restriction on $p+q-m+1$ in \eqref{eq:low}.
Since $q>0$, every positive constant solves \eqref{eq:pde}.
The case $q=0$ is not included; with the convention $|Du|^0=1$, a positive constant would no longer be a solution.

We next recall some of the main developments for elliptic equations with gradient terms. When $m=2$ and the reaction depends only on the gradient, equation~\eqref{eq:pde} reduces to
\begin{equation}\label{eq:pde2}
 -\Delta u=|Du|^q\quad\text{in }\R^n.
\end{equation}
Lions \cite{Lions1985} proved that every classical entire solution is constant when $q>1$. Souplet and Zhang \cite{SoupletZhang2006} studied more general inhomogeneous Hamilton--Jacobi equations. For the quasilinear equation with gradient absorption, Bidaut-V\'eron, Garc\'ia-Huidobro, and V\'eron \cite{BVGHV2014} established local gradient estimates, universal bounds, and Liouville theorems. Bidaut-V\'eron \cite{BidautVeron2021} later treated the super-natural source-gradient range for the $m$-Laplacian and obtained both Liouville results and asymptotic information.

A related line of work concerns equations in which the reaction and gradient terms are added rather than multiplied. Bidaut-V\'eron, Garc\'ia-Huidobro, and V\'eron \cite{BVGHV2020A,BVGHV2020R} studied a priori estimates, Liouville properties, and radial solutions for semilinear equations of the form
\[
 -\Delta u=u^p+M|Du|^q.
\]
Filippucci, Sun, and Zheng \cite{FilippucciSunZheng2024} obtained corresponding estimates and Liouville theorems for quasilinear equations involving the $m$-Laplacian. In this direction, Ma, Wu, and Zhang \cite{MaWuZhang2024} proved a Liouville theorem in the critical and subcritical cases for
\[
 \Delta u+Nu^p+M|Du|^q=0,
 \qquad q=\frac{2p}{p+1},
\]
using a differential identity and Young's inequality. Wu and Zhang \cite{WuZhang2025} extended this approach to the quasilinear equation
\[
 \Delta_m u+u^p+M|Du|^q=0,
 \qquad q=\frac{mp}{p+1}.
\]
These works show that integral identities provide an effective alternative to Bernstein-type arguments in critical problems with gradient terms.

We now turn to product-type reactions. When $m=2$, equation~\eqref{eq:pde} becomes
\begin{equation}\label{eq:pde4}
 -\Delta u=u^p|Du|^q\quad\text{in }\R^n.
\end{equation}
Bidaut-V\'eron, Garc\'ia-Huidobro, and V\'eron \cite{BV-GH-V2019} studied \eqref{eq:pde4} for $p\geqslant0$, $0\leqslant q<2$, and $p+q>1$. They proved the Liouville property in part of the parameter range and determined the exact threshold for the existence of nonconstant radial positive entire solutions. More precisely, when $0\leqslant q<1$, such radial solutions exist precisely when
\[
 (n-2)p+(n-1)q
 \geqslant n+\frac{2-q}{1-q}.
\]
Their results left the natural question of whether every positive entire solution must be constant below this threshold, including the part not covered by their Liouville theorem.

Ma and Wu \cite{MaWu2026} later formulated this question explicitly
and answered it completely when
\[
0\leqslant q\leqslant\frac{1}{n-1}.
\]
For larger $q$, they also improved the Liouville region obtained in
\cite{BV-GH-V2019}. More precisely, they proved that every positive
entire solution is constant, under the assumptions $p\ge0$ and
$p+q>1$, whenever
\[
p^2+
\left(
\frac{n-1}{n-2}q-\frac n2-\frac{3}{(n-2)^2}
\right)p
+\frac{1-(n-1)q}{(n-2)^2}<0,
\]
with
\[
\frac{1}{n-1}<q<1
\quad\text{if } n=3,
\qquad
\frac{1}{n-1}<q<2
\quad\text{if } n\ge4.
\]
When $m=2$, condition~\eqref{eq:low} is equivalent to
\[
p(n-2)+(n-1)q
<
n+\frac{2-q}{1-q}.
\]
Thus Theorem~\ref{thm:main} answers the above question for every
$0<q<1$. In particular, it completes the range left open by
\cite{BV-GH-V2019} and \cite{MaWu2026}, and it does not require the
assumptions $p\ge0$ and $p+q>1$.

Several recent papers and preprints have further developed the
semilinear product-gradient theory. Lu \cite{Lu2026} studied
positive solutions of \eqref{eq:pde4} for $q\geqslant0$ and
$p\in\mathbb R$ by deriving a new differential identity for
$|\nabla\log u|^2$. He obtained substantially larger Liouville
regions than those previously known. In particular, when $n\geqslant3$
and
\[
0\leqslant q\leqslant
1-\frac{1}{\sqrt{n-1}},
\]
his Liouville theorem reaches the threshold determined by the
existence of radial solutions and is therefore optimal in this
range. He also proved constancy in several further regions,
including
\[
q\geqslant\frac{5}{3}
\]
and
\[
p+q\leqslant\frac{n+2}{n-2},
\qquad n\geqslant3,
\]
together with additional parameter ranges beyond these simple
conditions. Moreover, for the parameter pairs covered by his
Liouville theorem, he established corresponding local gradient
estimates and Harnack inequalities. His arguments also extend to
Riemannian manifolds under lower bounds on the Ricci curvature.

Dou, Shi, Wu, and Zhu \cite{DouShiWuZhu2026} used the invariant
tensor technique to establish a differential identity on complete
noncompact manifolds and obtained classification results in the
second critical case with dimensions $n\leqslant5$. Very recently,
Liang and Zhang \cite{LiangZhang2026} derived Liouville theorems
and universal local estimates for \eqref{eq:pde4} through a new
linearized differential inequality.

For the general $m$-Laplacian, Filippucci \cite{Filippucci2009}
proved nonexistence results that include the first subcritical range
\[
(n-m)p+(n-1)q<(m-1)n.
\]
Chang, Hu, and Zhang \cite{ChangHuZhang2022} combined Bernstein
estimates with blow-up arguments to obtain Liouville theorems in a
larger part of the parameter space. For the second critical case, the authors recently obtained a complete
classification of positive entire solutions in \cite{Wu-Y-Z2026},
drawing on ideas from the quasilinear classification and Harnack
framework recently developed by Zhang \cite{ZhangYi2026}. Filippucci, Sun, and Zheng
\cite{FilippucciSunZheng2024} treated related equations with separate
source and gradient terms.

On complete noncompact Riemannian manifolds, Sun, Xiao, and Xu
\cite{SunXiaoXu2022} established sharp volume-growth criteria for the
inequality
\[
\Delta_m u+u^p|Du|^q\leqslant0.
\]
He, Hu, and Wang \cite{HeHuWang2026} and He and Wang
\cite{HeWang2026} used Nash--Moser iteration to obtain logarithmic
gradient estimates and Liouville properties for broad classes of
quasilinear equations on manifolds. Lu and Zhu \cite{LuZhu} proved a
general product-gradient theorem which, in the power case, covers
$q>m$ and every real $p$. A recent preprint of Bhakta, Biswas, and
Filippucci \cite{BhaktaBiswasFilippucci2025} extends the discussion
to nonhomogeneous $(p,q)$-Laplacian operators with product-type or
additive gradient reactions.

The main argument of this paper concerns $0<q<m$. We first reduce the problem to a lower bound for a function $K_\beta(t)$ of one variable obtained from an auxiliary function of $u$ and $|Du|$. We then construct the required function explicitly. The maximum-principle argument is local at a point where $|Du|>0$, while a standard lower bound for positive $m$-superharmonic functions controls the behavior at infinity. Section~\ref{sec:criterion} establishes the criterion used in the proof, and Section~\ref{sec:construction} constructs the auxiliary function and proves Theorem~\ref{thm:main}. In Section~\ref{sec:upper}, we treat $q=m$ by a direct change of the dependent variable and recall the theorem of Lu and Zhu for $q>m$. This proves Corollary~\ref{cor:complete}.

\section{A criterion for constancy}\label{sec:criterion}

Throughout this section, assume that $0<q<m$. Recall that
\[
 \Delta_m u=\operatorname{div}\bigl(|Du|^{m-2}Du\bigr).
\]
All differential calculations will be made on the regular set
\[
 \Omega_{\mathrm{reg}}=\{x\in\R^n:|Du(x)|>0\}.
\]
Because $u$ is positive and $u,Du$ are continuous, the right-hand side of \eqref{eq:pde} is locally bounded. The interior regularity theory for the $m$-Laplacian therefore yields
$u\in C^{1,\alpha}_{\mathrm{loc}}$ for some $\alpha>0$; see \cite{DiBenedetto,Tolksdorf}.

We explain why higher derivatives may be used on $\Omega_{\mathrm{reg}}$. For $\xi\ne0$,
\[
 \frac{\partial}{\partial\xi_j}\bigl(|\xi|^{m-2}\xi_i\bigr)
 =|\xi|^{m-2}\left(\delta_{ij}+(m-2)\frac{\xi_i\xi_j}{|\xi|^2}\right).
\]
On every compact subset of $\Omega_{\mathrm{reg}}$, both $u$ and $|Du|$ are bounded above and bounded away from zero. Hence the maps
\[
 \xi\longmapsto |\xi|^{m-2}\xi,
 \qquad
 (s,\xi)\longmapsto s^p|\xi|^q
\]
are smooth on the relevant ranges, and the matrix displayed above is uniformly positive definite. Local difference-quotient estimates followed by Schauder bootstrapping imply
\[
 u\in C^\infty(\Omega_{\mathrm{reg}}).
\]
Thus every differentiated identity below is classical on $\Omega_{\mathrm{reg}}$. No second derivative across $\{|Du|=0\}$ is used.

\subsection{Differential identities on the regular set}
On $\Omega_{\mathrm{reg}}$, introduce
\begin{equation*}\label{eq:basic}
 r=|Du|,\qquad h=\frac{r^2}{u^2},\qquad
 t=u^{p+1}r^{q-m},\qquad
 \ell=p+q-m+1.
\end{equation*}
Here $h=|D\log u|^2$ is the squared logarithmic gradient, while $\ell$ records the scaling of $t$. These quantities satisfy
\begin{equation*}\label{eq:scaling}
 t=u^\ell h^{-(m-q)/2}.
\end{equation*}

For later calculations, define
\[
 A_{ij}=\delta_{ij}+(m-2)\frac{u_i u_j}{r^2}.
\]
We shall use the normalized linearized operator and its principal second-order part:
\begin{equation}\label{eq:operators}
 \mathcal P\varphi=r^{2-m}\partial_i\bigl(r^{m-2}A_{ij}\varphi_j\bigr),
 \qquad
 \mathcal L\varphi=A_{ij}\varphi_{ij}.
\end{equation}
The operator $\mathcal P$ is convenient when the equation is differentiated, while $\mathcal L$ is the operator used at a maximum point.

Fix a point in $\Omega_{\mathrm{reg}}$ and choose an orthonormal frame with $e_1=Du/r$. Write
\begin{equation*}\label{eq:hessian}
 X=\frac{u u_{11}}{r^2},\qquad
 Z_\alpha=\frac{u u_{1\alpha}}{r^2},\qquad
 T_{\alpha\beta}=\frac{u u_{\alpha\beta}}{r^2}
 \quad(2\leqslant\alpha,\beta\leqslant n).
\end{equation*}
Thus $X$ is the normalized second derivative in the gradient direction, $Z=(Z_2,\ldots,Z_n)$ contains the mixed derivatives, and $T=(T_{\alpha\beta})_{2\leqslant\alpha,\beta\leqslant n}$ is the tangential block. We use
\[
 |Z|^2=\sum_{\alpha=2}^n Z_\alpha^2,
 \qquad
 |T|^2=\sum_{\alpha,\beta=2}^n T_{\alpha\beta}^2.
\]
Since
\[
 \Delta_m u=r^{m-2}A_{ij}u_{ij},
\]
multiplying equation~\eqref{eq:pde} by $u/r^m$ yields
\begin{equation}\label{eq:trace}
 (m-1)X+\operatorname{tr}T=-t.
\end{equation}
The gradients of the two basic logarithmic quantities are
\begin{equation}\label{eq:gradients}
 \frac{u}{r}D\log u=e_1,\qquad
 \frac{u}{r}D\log h=2(X-1,Z).
\end{equation}
In the second identity, $(X-1,Z)$ denotes the vector with first component $X-1$ and tangential components $Z_2,\ldots,Z_n$.

\begin{lemma}\label{lem:logs}
On $\Omega_{\mathrm{reg}}$, we have
\begin{align}
 \frac{\mathcal P\log u}{h}&=-(m-1)(1+t),\label{eq:logu}\\
 \frac{\mathcal P\log h}{h}
 &=2\bigl[-(m-1)X^2+(m-2)|Z|^2+|T|^2
             -pt-qtX+(m-1)(1+t)\bigr].\label{eq:logh}
\end{align}
\end{lemma}
\begin{proof}
Since $A Du=(m-1)Du$,
\[
 \mathcal P u=(m-1)r^{2-m}\Delta_m u.
\]
Thus the equation~\eqref{eq:pde} implies \eqref{eq:logu}. Differentiating \eqref{eq:pde} on $\Omega_{\mathrm{reg}}$ yields
\begin{align*}
 \mathcal P(r^2)=2A_{ij}u_{ki}u_{kj}+2r^{2-m}u_k\partial_k\Delta_m u,
 \qquad\partial_k\Delta_m u=-p u^{p-1}u_k r^q-q u^p r^{q-1}r_k.
\end{align*}
It follows that
\[
 \frac{\mathcal P\log r^2}{h}
 =2\bigl[-(m-1)X^2+(m-2)|Z|^2+|T|^2-pt-qtX\bigr].
\]
Since $\log h=\log r^2-2\log u$, subtracting twice \eqref{eq:logu} proves \eqref{eq:logh}.
\end{proof}

\subsection{The auxiliary function and the function $K_\beta$}
Let $\Phi:[0,\infty)\to\R$ be a smooth function and let $\beta$ be a constant to be determined later. Define the auxiliary function
\begin{equation}\label{eq:auxiliary}
 G=u^{-\frac{m-q}{2}\beta}h^{\frac{m-q}{2}}e^{\Phi(t)}.
\end{equation}
Rather than prescribing $\Phi$ directly, it is convenient to use
\begin{equation*}\label{eq:slope}
 \omega(t)=t\Phi'(t)=\frac{d}{d\log t}\Phi(t).
\end{equation*}
Since
\[
 \log t=\ell\log u-\frac{m-q}{2}\log h,
\]
the differential of $\log G$ can be written as
\[
 d\log G=U(t)\,d\log h+V(t)\,d\log u,
\]
where
\begin{equation}\label{eq:UV}
 U(t)=\frac{m-q}{2}\bigl(1-\omega(t)\bigr),\qquad
 V(t)=-\frac{m-q}{2}\beta+\ell\omega(t).
\end{equation}
The functions $U$ and $V$ simply record the two coefficients in this differential. The second derivative of $\Phi(t)$ with respect to the variables $(\log u,\log h)$ is
\[
 t\omega'(t)\,d(\log t)\otimes d(\log t).
\]
The chain rule for $\mathcal P$ therefore becomes
\begin{equation}\label{eq:chain}
 \mathcal P\log G
 =U\mathcal P\log h+V\mathcal P\log u
       +t\omega'(t)\,\langle A D\log t,D\log t\rangle.
\end{equation}
We shall always require $1-\omega(t)>0$, so that $U(t)>0$.

Equation~\eqref{eq:gradients} implies
\[
\frac{u}{r}D\log G
=
2U(X-1,Z)+V e_1.
\]
Since $U(t)>0$, the condition $D\log G=0$ determines $Z$ and $X$ uniquely:
\[
Z=0,\qquad
X=1-\frac{V(t)}{2U(t)}=:X_*(t).
\]
Thus, at a critical point of $G$, the quantity
\[
X=\frac{u\,u_{11}}{|Du|^2}
\]
is completely determined by $t$.

At such a point, the preceding identities also yield
\[
D\log t
=
\frac{\ell-\frac{m-q}{2}\beta}{1-\omega(t)}
D\log u,
\]
and therefore
\[
\frac{\langle A D\log t,D\log t\rangle}{h}
=
(m-1)
\frac{\left(\ell-\frac{m-q}{2}\beta\right)^2}
{\bigl(1-\omega(t)\bigr)^2}.
\]

It remains to deal with the tangential part $T$ of the Hessian.
By \eqref{eq:trace} and the identity $X=X_*(t)$,
\[
\operatorname{tr}T
=
-t-(m-1)X_*(t).
\]
For every symmetric $(n-1)\times(n-1)$ matrix with this trace,
\[
|T|^2
\geqslant
\frac{(\operatorname{tr}T)^2}{n-1}
=
\frac{\bigl(t+(m-1)X_*(t)\bigr)^2}{n-1},
\]
with equality when $T$ is a multiple of the identity.

Consequently, among all Hessian matrices satisfying
$D\log G=0$ and \eqref{eq:trace}, the minimum of
$\mathcal P\log G/h$ depends only on $t$.
We denote this minimum by $K_\beta(t)$:
\begin{align}
K_\beta(t)
={}&
2U(t)\left[
-(m-1)X_*(t)^2
+\frac{\bigl(t+(m-1)X_*(t)\bigr)^2}{n-1}
\right]\notag\\
&+2U(t)\bigl[
-pt-qtX_*(t)+(m-1)(1+t)
\bigr]\notag\\
&-(m-1)V(t)(1+t)+(m-1)t\omega'(t)
\frac{\left(\ell-\frac{m-q}{2}\beta\right)^2}
{\bigl(1-\omega(t)\bigr)^2}.
\label{eq:kernel}
\end{align}
In particular, at every regular point where $D\log G=0$, we have
\[
\frac{\mathcal P\log G}{h}\geqslant K_\beta(t).
\]
Hence, if $K_\beta(t)>0$ for all $t\geqslant0$, then
$\mathcal P\log G>0$ at every critical point of $G$ in
$\Omega_{\mathrm{reg}}$.
The differential problem is therefore reduced to
the positivity of the one-variable function $K_\beta$. When $\omega$ depends on an additional parameter, we include that
parameter in the notation for $K_\beta$.

\subsection{From the estimate for $K_\beta$ to constancy}
Since $u$ is a positive entire weakly $m$-superharmonic function,
the standard lower bound
\begin{equation}\label{eq:lower}
u(x)\geqslant c(1+|x|)^{-\frac{n-m}{m-1}}
\qquad\text{for all }x\in\mathbb R^n
\end{equation}
holds for some constant $c>0$ depending on $u$; see
\cite{SerrinZou}.

\begin{proposition}\label{prop:criterion}
Assume that a smooth function $\Phi:[0,\infty)\to\R$ and a constant
$\beta$ can be chosen so that
\begin{equation}\label{eq:beta}
 0\leqslant\beta<\beta_*:=\frac{2(m-1)}{n-m},
\end{equation}
and, for every $t\geqslant0$,
\begin{equation}\label{eq:bounds}
 0\leqslant\omega(t)\leqslant1-\rho,\qquad
 |t\omega'(t)|+|\Phi(t)|\leqslant C,\qquad
 K_\beta(t)\geqslant\kappa(1+t)^2,
\end{equation}
where $\rho>0$, $\kappa>0$, and $C<\infty$. Then every positive
$C^1_{\mathrm{loc}}$ weak solution of \eqref{eq:pde} is constant.
\end{proposition}

\begin{proof}
At a general point of $\Omega_{\mathrm{reg}}$, the gradient of $G$
need not vanish. We first estimate how much
$\mathcal P\log G/h$ can differ from the value used in the
definition of $K_\beta(t)$.

Write the tangential Hessian block as
\[
 T=\frac{\operatorname{tr}T}{n-1}I+T^\circ,
 \qquad \operatorname{tr}T^\circ=0.
\]
By \eqref{eq:gradients}, \eqref{eq:UV}, and the definition of
$X_*(t)$,
\begin{equation}\label{eq:Df}
 \frac{u}{r}D\log G
 =2U\bigl(X-X_*(t),Z\bigr),
 \qquad
 \frac{|D\log G|^2}{h}
 =4U^2\left(\bigl(X-X_*(t)\bigr)^2+|Z|^2\right).
\end{equation}
The trace relation \eqref{eq:trace} and the identity
\[
D\log t=\ell D\log u-\frac{m-q}{2}D\log h
\]
imply
\begin{align}
 |T|^2
 &=\frac{\bigl[t+(m-1)X\bigr]^2}{n-1}+|T^\circ|^2,
 \label{eq:tracefree}\\
 \frac{\langle A D\log t,D\log t\rangle}{h}
 &=(m-1)\left(
 \frac{\ell-\frac{m-q}{2}\beta}{1-\omega(t)}
 -(m-q)\bigl(X-X_*(t)\bigr)
 \right)^2
 +(m-q)^2|Z|^2.
 \label{eq:offgradient}
\end{align}
The assumptions in \eqref{eq:bounds} ensure that
\[
 \frac{m-q}{2}\rho\leqslant U(t)\leqslant\frac{m-q}{2},
 \qquad |X_*(t)|\leqslant C,
 \qquad
 \left|
 \frac{\ell-\frac{m-q}{2}\beta}{1-\omega(t)}
 \right|\leqslant C.
\]
Substituting \eqref{eq:tracefree} and \eqref{eq:offgradient} into
\eqref{eq:chain}, and subtracting the expression obtained when
\[
X=X_*(t),\qquad Z=0,\qquad T^\circ=0,
\]
we find
\begin{equation}\label{eq:Pestimate}
 \frac{\mathcal P\log G}{h}
 \geqslant K_\beta(t)
 -C(1+t)\bigl|X-X_*(t)\bigr|
 -C\left(\bigl(X-X_*(t)\bigr)^2+|Z|^2\right).
\end{equation}
The term involving $T^\circ$ is nonnegative and has therefore been
omitted from the right-hand side.

We next replace $\mathcal P$ with $\mathcal L$, since $\mathcal L$
is the operator used at a maximum point. Expanding the derivative
in \eqref{eq:operators},
\[
 (\mathcal P-\mathcal L)\varphi
 =r^{2-m}\partial_i\bigl(r^{m-2}A_{ij}\bigr)\varphi_j.
\]
In the frame $e_1=Du/r$, the coefficient of $\varphi_1$ before
using the equation is
\[
 \frac{m-2}{r}\left(
 \sum_{\alpha=2}^n u_{\alpha\alpha}+(m-1)u_{11}
 \right),
\]
and the coefficient of $\varphi_\alpha$ is
\[
\frac{2(m-2)}{r}u_{1\alpha}.
\]
The trace relation \eqref{eq:trace} therefore yields
\begin{equation}\label{eq:drift}
 (\mathcal P-\mathcal L)\varphi
 =(m-2)\frac{r}{u}
 \left(
 -t\varphi_1+2\sum_{\alpha=2}^n Z_\alpha\varphi_\alpha
 \right).
\end{equation}
Taking $\varphi=\log G$ and using \eqref{eq:Df}, we obtain
\[
 |(\mathcal P-\mathcal L)\log G|
 \leqslant Ch\left(
 t\bigl|X-X_*(t)\bigr|+|Z|^2
 \right).
\]
Combining this estimate with \eqref{eq:Pestimate}, we may replace
$\mathcal P$ by $\mathcal L$ on the left-hand side, after changing
the constant $C$.

The elementary inequality
\[
 C(1+t)\bigl|X-X_*(t)\bigr|
 \leqslant \frac{\kappa}{2}(1+t)^2
 +C\bigl(X-X_*(t)\bigr)^2
\]
and the first identity in \eqref{eq:Df} now imply
\begin{equation}\label{eq:local}
 \mathcal L\log G
 \geqslant\frac{\kappa}{2}h(1+t)^2-C|D\log G|^2
 \geqslant c u^\beta G^{2/(m-q)}-C|D\log G|^2.
\end{equation}
For the second inequality, we used
\[
 u^\beta G^{2/(m-q)}
 =h e^{2\Phi(t)/(m-q)}
\]
and the upper bound for $\Phi$.

We now apply \eqref{eq:local} at a maximum point. From
\eqref{eq:auxiliary},
\[
 G
 =u^{-\frac{m-q}{2}(\beta+2)}
 |Du|^{m-q}e^{\Phi(t)}.
\]
On each compact set, $u$ is bounded above and bounded away from
zero, while $\Phi$ is bounded. Since $m-q>0$, it follows that
\[
 0<G\leqslant C|Du|^{m-q}
 \qquad\text{on }\Omega_{\mathrm{reg}}.
\]
Hence $G$ extends continuously to $\{|Du|=0\}$ by assigning the
value zero there.

Assume that $\Omega_{\mathrm{reg}}$ is nonempty. Choose $R$ large
enough that $B_R=B_R(0)$ contains a point of
$\Omega_{\mathrm{reg}}$, and define
\[
 \eta_R(x)=1-\frac{|x|^2}{R^2},
 \qquad
 W_R(x)=\eta_R(x)^{m-q}G(x).
\]
The function $W_R$ is continuous on $\overline B_R$, vanishes on
$\partial B_R$ and on $\{|Du|=0\}$, and is positive somewhere.
It therefore attains a positive maximum at a point
\[
x_R\in B_R\cap\Omega_{\mathrm{reg}}.
\]
At this point,
\[
 D\log G=-(m-q)D\log\eta_R,
 \qquad
 \mathcal L\log G\leqslant-(m-q)\mathcal L\log\eta_R.
\]
Since the eigenvalues of $A$ are bounded above and below by positive
constants depending only on $m$, direct differentiation of
$\eta_R$ shows that
\[
 |D\log G|^2\leqslant\frac{C}{R^2\eta_R^2},
 \qquad
 \mathcal L\log G\leqslant\frac{C}{R^2\eta_R^2}
 \quad\text{at }x_R.
\]
Substituting these inequalities into \eqref{eq:local}, we obtain
\[
 u(x_R)^\beta G(x_R)^{2/(m-q)}
 \leqslant\frac{C}{R^2\eta_R(x_R)^2}.
\]
Because
\[
W_R(x_R)=\eta_R(x_R)^{m-q}G(x_R),
\]
the powers of $\eta_R(x_R)$ cancel, and
\begin{equation}\label{eq:decay}
 \left(\max_{\overline B_R}W_R\right)^{2/(m-q)}
 \leqslant C R^{-2}u(x_R)^{-\beta}.
\end{equation}
The lower bound \eqref{eq:lower} and the fact that $|x_R|<R$ imply,
for $R\geqslant1$,
\[
 u(x_R)^{-\beta}
 \leqslant C R^{\beta(n-m)/(m-1)}.
\]
Thus \eqref{eq:decay} becomes
\[
 \left(\max_{\overline B_R}W_R\right)^{2/(m-q)}
 \leqslant C R^{-2+\beta(n-m)/(m-1)}.
\]
The exponent on the right is negative by \eqref{eq:beta}. If
$x\in\Omega_{\mathrm{reg}}$ is fixed, then
\[
W_R(x)\leqslant\max_{\overline B_R}W_R
\qquad\text{and}\qquad
\eta_R(x)\to1
\]
as $R\to\infty$. Hence the preceding estimate forces $G(x)=0$,
which contradicts the definition of $G$ on
$\Omega_{\mathrm{reg}}$. Therefore $\Omega_{\mathrm{reg}}$ is empty,
and $u$ is constant.
\end{proof}

\section{Construction of the auxiliary function}\label{sec:construction}

We continue to assume that $0<q<m$. Proposition~\ref{prop:criterion} shows that it is enough to choose $\Phi$ and $\beta$ so that the function $K_\beta(t)$ satisfies the estimates in \eqref{eq:bounds}. Recall that
\[
 \omega(t)=t\Phi'(t).
\]
We begin with
\[
 \beta_* = \frac{2(m-1)}{n-m},
\]
for which $K_{\beta_*}(0)=0$. We first choose $\omega$ so that
$K_{\beta_*}(t)>0$ for every $t>0$. We then change $\omega$ only for large $t$ so that
\[
 \Phi(t)=\int_0^t\frac{\omega(s)}{s}\,ds
\]
is bounded. Finally, we choose $\beta<\beta_*$ close to $\beta_*$ so that $K_\beta(0)>0$ as well. These steps produce the lower bound required in Proposition~\ref{prop:criterion}.

\subsection{From positivity for $t>0$ to a uniform lower bound}

\begin{lemma}\label{lem:completion}
Suppose that $\omega_0$ is smooth on $[0,\infty)$ and satisfies
\begin{gather}
 \omega_0(0)=0,\qquad
 0\leqslant\omega_0(t)\leqslant1-\rho,\qquad |t\omega_0'(t)|\leqslant C,
 \label{eq:initialslope}\\
 K_{\beta_*}(t)>0\quad(t>0),\qquad
 K_{\beta_*}(0)=0,\qquad K_{\beta_*}'(0)>0.
 \label{eq:initialkernel}
\end{gather}
Then there is a smooth function $\omega$ that agrees with $\omega_0$ on a bounded interval and is zero for all sufficiently large $t$. For this choice of $\omega$, one can take $\beta<\beta_*$ sufficiently close to $\beta_*$ so that
\[
 \Phi(t)=\int_0^t\frac{\omega(s)}{s}\,ds
\]
satisfies all the assumptions of Proposition~\ref{prop:criterion}.
\end{lemma}

\begin{proof}
Choose a smooth function $\chi:\R\to[0,1]$ such that $\chi=1$ on $(-\infty,0]$ and $\chi=0$ on $[1,\infty)$. For $T_0>1$, define
\[
 \omega(t)=\chi\bigl(\log(t/T_0)\bigr)\omega_0(t)
 \quad(t>0),\qquad \omega(0)=0.
\]
Then $\omega=\omega_0$ for $t\leqslant T_0$ and $\omega=0$ for $t\geqslant eT_0$. Moreover,
\[
 0\leqslant\omega(t)\leqslant1-\rho,
 \qquad |t\omega'(t)|\leqslant C,
\]
where the constant is independent of $T_0$.

In formula~\eqref{eq:kernel}, the coefficient of $t^2$ is
\[
 \frac{(m-q)(1-\omega(t))}{n-1}.
\]
The remaining terms are bounded below by $-Ct-C$, uniformly in $T_0$ and for $\beta$ close to $\beta_*$. Hence
\begin{equation}\label{eq:tail}
 K_\beta(t)\geqslant\frac{(m-q)\rho}{n-1}t^2-Ct-C.
\end{equation}
If $T_0$ is sufficiently large, the right-hand side is positive for $t\geqslant T_0$. On $[0,T_0]$, the new function $\omega$ agrees with $\omega_0$. Thus, at $\beta=\beta_*$,
\[
 K_{\beta_*}(t)>0\qquad(t>0).
\]
Since $\omega$ is smooth and $\omega(0)=0$, the integral defining $\Phi$ is smooth at $t=0$. Since $\omega(t)=0$ for $t\geqslant eT_0$, the function $\Phi$ is bounded and constant for all sufficiently large $t$.

We next choose $\beta<\beta_*$. At $t=0$, formula~\eqref{eq:kernel} becomes
\begin{equation}\label{eq:origin}
 K_\beta(0)
 =\frac{(m-q)(m-1)(\beta+2)[2(m-1)-(n-m)\beta]}{4(n-1)}.
\end{equation}
This number is positive whenever $0\leqslant\beta<\beta_*$. By \eqref{eq:initialkernel} and continuity with respect to $(t,\beta)$, the derivative $K_\beta'(t)$ remains positive on a fixed interval near $t=0$ when $\beta$ is sufficiently close to $\beta_*$. Formula~\eqref{eq:origin} then proves positivity near $t=0$. Positivity on the remaining bounded interval follows from continuity, while \eqref{eq:tail} controls all sufficiently large $t$. Therefore, for some $\kappa>0$,
\[
 K_\beta(t)\geqslant\kappa(1+t)^2\qquad(t\geqslant0).
\]
\end{proof}

\subsection{The case $\ell\leqslant\ell_1$}

Define
\begin{equation*}\label{eq:ellone}
 \ell_1=\frac{(m-1)(m-q)}{n-m}.
\end{equation*}
Assume first that $\ell\leqslant\ell_1$. Choose $\Phi=0$. Then $\omega(t)=0$ and $t\omega'(t)=0$, and formula~\eqref{eq:kernel} becomes
\begin{align*}
 \frac{K_\beta(t)}{m-q}
 ={}&\frac{t^2}{n-1}
 +\left[\frac{2(m-1)}{n-1}-\ell
       +\beta\left(\frac{m-1}{n-1}+\frac{m-1-q}{2}\right)\right]t\notag\\
 &+\frac{(m-1)(\beta+2)[2(m-1)-(n-m)\beta]}{4(n-1)}.
 \label{eq:constant}
\end{align*}
At $\beta=\beta_*$, the constant term is zero, while the coefficient of $t$ is
\[
 \ell_1-\ell+\frac{m-1}{n-m}>0.
\]
If $\beta<\beta_*$ is sufficiently close to $\beta_*$, both the constant term and the coefficient of $t$ are positive. The coefficient of $t^2$ is also positive. Thus
\[
 K_\beta(t)\geqslant\kappa(1+t)^2\qquad(t\geqslant0),
\]
and Proposition~\ref{prop:criterion} shows that $u$ is constant.

\subsection{Choosing $\omega$ when $\ell>\ell_1$}

Assume now that
\begin{equation*}\label{eq:mu}
 \delta_\ell:=\ell-\ell_1>0.
\end{equation*}
The expression
\begin{equation*}\label{eq:Dq}
 \mathcal D(q):=n(m-1)-(n-1)q
      =(m-1)+(n-1)(m-1-q)
\end{equation*}
appears repeatedly below, so we introduce this notation once. For
\[
 0<\theta<\frac{2}{\max\{m-1,\mathcal D(q)\}},
\]
define
\begin{equation*}\label{eq:Lambda}
 \Lambda(\theta)=\ell_1+
 \frac{(m-1)\bigl(2/(n-m)+\theta\bigr)}{2-(m-1)\theta}.
\end{equation*}
A direct differentiation shows that
\[
 \Lambda'(\theta)
 =\frac{2(m-1)(n-1)}{(n-m)[2-(m-1)\theta]^2}>0.
\]

If $0<q<m-1$, then $\mathcal D(q)>m-1$ and
\begin{equation}\label{eq:Lambdaendpoint}
\Lambda\left(\frac{2}{\mathcal D(q)}\right)
=\frac{(m-1)(m-q)^2}{(n-m)(m-1-q)}.
\end{equation}
Thus the first condition in Theorem~\ref{thm:main} allows us to
choose $\theta$ below $2/\mathcal D(q)$ such that
$\ell<\Lambda(\theta)$.

If $m-1\leqslant q<m$, then $\mathcal D(q)\leqslant m-1$ and
$\Lambda(\theta)\to\infty$ as $\theta\uparrow2/(m-1)$. The assumptions of Theorem~\ref{thm:main} therefore allow us to choose $\theta$ such that
\begin{equation}\label{eq:thetachoice}
 0<\theta<\frac{2}{\max\{m-1,\mathcal D(q)\}},
 \qquad \ell<\Lambda(\theta).
\end{equation}
This choice is possible in both ranges of $q$; only the upper endpoint of the interval for $\theta$ is different.

Define
\begin{equation*}\label{eq:d}
 \lambda_\theta:=\frac{(m-q)\theta}{2\delta_\ell}.
\end{equation*}
We first consider the simple function
\[
 \omega_0(t)=\frac{\lambda_\theta t}{1+\lambda_\theta t}.
\]
Let $K_{0,\beta_*}$ denote the function $K_{\beta_*}$ corresponding to this choice. Substitution into \eqref{eq:kernel} yields
\begin{align}
 K_{0,\beta_*}(t)
 =\frac{\lambda_\theta t}{1+\lambda_\theta t}
 \Biggl\{
 \frac{\delta_\ell[2-(m-1)\theta]}{\theta}
 [\Lambda(\theta)-\ell]
 +\frac{\delta_\ell^2[2-(m-1)\theta]
 [2-\mathcal D(q)\theta]}
 {(n-1)(m-q)\theta^2}\,\lambda_\theta t
 \Biggr\}.
 \label{eq:outer}
\end{align}
Both coefficients inside the braces are positive by \eqref{eq:thetachoice}. Hence
\[
 K_{0,\beta_*}(t)>0\qquad(t>0).
\]
However, $\omega_0(t)\to1$ as $t\to\infty$, whereas Proposition~\ref{prop:criterion} requires $\omega$ to remain a fixed positive distance below $1$. We therefore replace $\omega_0$ by the following nearby function.

\subsection{A choice of $\omega$ that remains below $1$}

Choose $0<\nu\leqslant1$ sufficiently small that
\begin{equation}\label{eq:nuchoice}
 (m-1)(n-1)(m-q)\theta^2\nu
 \leqslant\frac12[2-(m-1)\theta]
                [2-\max\{\mathcal D(q),0\}\theta].
\end{equation}
The right-hand side is positive by \eqref{eq:thetachoice}. For $0<\varepsilon\leqslant1$, define
\begin{equation}\label{eq:explicit}
 1-\omega_\varepsilon(t)
 =\left(
 \frac{(1+\lambda_\theta t)^{-\nu}+\varepsilon}{1+\varepsilon}
 \right)^{1/\nu}.
\end{equation}
The number $\varepsilon$ keeps $1-\omega_\varepsilon$ positive for all $t$, while the small exponent $\nu$ controls the additional term that appears in the calculation below. Formula~\eqref{eq:explicit} shows that
\begin{equation}\label{eq:epsbounds}
 \omega_\varepsilon(0)=0,\qquad
 \left(\frac{\varepsilon}{1+\varepsilon}\right)^{1/\nu}
 \leqslant1-\omega_\varepsilon(t)\leqslant1.
\end{equation}
For each fixed $\varepsilon>0$, the function $\omega_\varepsilon$ is smooth on $[0,\infty)$. As $\varepsilon\downarrow0$, it converges smoothly on every bounded $t$-interval to $\omega_0$. We shall prove that, for every sufficiently small $\varepsilon>0$, the corresponding function $K_{\varepsilon,\beta_*}(t)$ is positive for all $t>0$.

To treat all $t\geqslant0$ in one calculation, set
\begin{equation}\label{eq:compact}
 \alpha=\frac{1}{1+\lambda_\theta t},\qquad
 z=\frac{\alpha}{(\alpha^\nu+\varepsilon)^{1/\nu}}.
\end{equation}
As $t$ increases from $0$ to $\infty$, the variable $\alpha$ decreases from $1$ to $0$. The second definition is chosen so that
\begin{equation}\label{eq:compactrelations}
 1-\omega_\varepsilon(t)
 =\frac{\alpha}{(1+\varepsilon)^{1/\nu}z},\qquad
 t\omega_\varepsilon'(t)
 =(1-\alpha)(1-\omega_\varepsilon(t))z^\nu.
\end{equation}
In particular,
\[
 0\leqslant t\omega_\varepsilon'(t)\leqslant1.
\]

At $\beta=\beta_*$, the relations in
\eqref{eq:compact} and \eqref{eq:compactrelations} imply
\begin{equation}\label{eq:scaled}
 \alpha t=\frac{1-\alpha}{\lambda_\theta},
 \qquad
 \alpha X_*(t)
 =
 \frac{\alpha(p+1)
 -\delta_\ell(1+\varepsilon)^{1/\nu}z}{m-q}.
\end{equation}
Although $t$ and $X_*(t)$ may be unbounded as $\alpha\to0$, the
two quantities in \eqref{eq:scaled} remain finite.

For the choice $\omega_\varepsilon$, let
$K_{\varepsilon,\beta_*}(t)$ denote the corresponding function in
\eqref{eq:kernel}. Since
\[
 \frac{\alpha^2}{1-\omega_\varepsilon(t)}>0,
\]
we define
\begin{equation}\label{eq:normalized}
 \widetilde K_\varepsilon(\alpha,z)
 :=
 \frac{\alpha^2}{1-\omega_\varepsilon(t)}
 K_{\varepsilon,\beta_*}(t).
\end{equation}
Substitution into \eqref{eq:kernel} yields
\begin{align}
 \widetilde K_\varepsilon(\alpha,z)
 ={}&(m-q)\left[
 -(m-1)\bigl(\alpha X_*(t)\bigr)^2
 +\frac{\bigl[\alpha t+(m-1)\alpha X_*(t)\bigr]^2}{n-1}
 \right]\notag\\
 &+(m-q)\left[
 -p\alpha(\alpha t)
 -q(\alpha t)\bigl(\alpha X_*(t)\bigr)
 +(m-1)\bigl(\alpha^2+\alpha(\alpha t)\bigr)
 \right]\notag\\
 &-(m-1)\left[
 \delta_\ell(1+\varepsilon)^{1/\nu}z-\ell\alpha
 \right](\alpha+\alpha t)+(m-1)\delta_\ell^2(1-\alpha)
 (1+\varepsilon)^{2/\nu}z^{\nu+2}.
 \label{eq:H}
\end{align}
Using the explicit formulas in \eqref{eq:scaled}, the right-hand
side extends continuously to
\[
 (\alpha,z,\varepsilon)\in[0,1]^2\times[0,1].
\]
In particular, it remains finite when $\alpha=0$ or $z=0$.
This is why we multiply
$K_{\varepsilon,\beta_*}(t)$ by the positive factor in
\eqref{eq:normalized}.

We need two special cases of this formula when $\varepsilon=0$. First, if $z=1$, then \eqref{eq:outer} becomes
\begin{align}
 \widetilde K_0(\alpha,1)
 =(1-\alpha)\Biggl\{
 \frac{\delta_\ell[2-(m-1)\theta]}{\theta}
 [\Lambda(\theta)-\ell] \,\alpha
 +\frac{\delta_\ell^2[2-(m-1)\theta]
 [2-\mathcal D(q)\theta]}
 {(n-1)(m-q)\theta^2}(1-\alpha)
 \Biggr\}.
 \label{eq:outerface}
\end{align}
This expression is positive for $0<\alpha<1$. Second, setting $\alpha=0$ in \eqref{eq:H} results in
\begin{align}
 \widetilde K_0(0,z)
 =\frac{\delta_\ell^2}{(n-1)(m-q)\theta^2}
 \Bigl\{&[2-(m-1)\theta z][2-\mathcal D(q)\theta z]\notag\\
 &-(m-1)(n-1)(m-q)\theta^2z^2(1-z^\nu)\Bigr\}.
 \label{eq:tailface}
\end{align}
For $0<z\leqslant1$,
\begin{equation*}\label{eq:smallerror}
 z^2(1-z^\nu)\leqslant\nu z^2(-\log z)\leqslant\nu,
\end{equation*}
and the same inequality holds at $z=0$ by continuity. Also,
\[
 [2-(m-1)\theta z][2-\mathcal D(q)\theta z]
 \geqslant[2-(m-1)\theta]
      [2-\max\{\mathcal D(q),0\}\theta].
\]
The choice of $\nu$ in \eqref{eq:nuchoice} therefore implies
\begin{equation}\label{eq:facegap}
 \widetilde K_0(0,z)
 \geqslant\frac{\delta_\ell^2[2-(m-1)\theta]
 [2-\max\{\mathcal D(q),0\}\theta]}
 {2(n-1)(m-q)\theta^2}>0
 \qquad(0\leqslant z\leqslant1).
\end{equation}

We now prove positivity for small $\varepsilon$ by contradiction. Suppose that there are sequences
\[
 \varepsilon_j\downarrow0,\qquad t_j>0,
 \qquad K_{\varepsilon_j,\beta_*}(t_j)\leqslant0.
\]
After passing to a subsequence, the variables in \eqref{eq:compact} satisfy
\[
 \alpha_j\to\alpha_0\in[0,1],\qquad
 z_j\to z_0\in[0,1].
\]
The factor used in \eqref{eq:normalized} is positive, so
\[
 \widetilde K_{\varepsilon_j}(\alpha_j,z_j)\leqslant0.
\]
If $\alpha_0=0$, continuity and \eqref{eq:facegap} lead to a contradiction. If $0<\alpha_0<1$, then \eqref{eq:compact} implies $z_j\to1$, and \eqref{eq:outerface} again leads to a contradiction.

The only remaining case is $\alpha_0=1$, which is equivalent to $t_j\to0$. Formulas~\eqref{eq:explicit} and \eqref{eq:kernel} show that $K_{\varepsilon,\beta_*}(t)$ is smooth in $(t,\varepsilon)$ near $(0,0)$ and satisfies
\[
 K_{\varepsilon,\beta_*}(0)=0.
\]
Formula~\eqref{eq:outer} shows that
\[
 \partial_tK_{0,\beta_*}(0)
 =\lambda_\theta
 \frac{\delta_\ell[2-(m-1)\theta]}{\theta}
 [\Lambda(\theta)-\ell]>0.
\]
By continuity, $\partial_tK_{\varepsilon,\beta_*}(t)>0$ when $t$ and $\varepsilon$ are sufficiently small. Hence
\[
 K_{\varepsilon,\beta_*}(t)
 =t\int_0^1\partial_tK_{\varepsilon,\beta_*}(st)\,ds>0
\]
for all sufficiently small positive $t$ and $\varepsilon$. This excludes the last case.

We have proved that every sufficiently small fixed $\varepsilon>0$ satisfies
\[
 K_{\varepsilon,\beta_*}(t)>0\quad(t>0),\qquad
 K_{\varepsilon,\beta_*}(0)=0,\qquad
 K_{\varepsilon,\beta_*}'(0)>0.
\]
The bounds in \eqref{eq:epsbounds} and
\eqref{eq:compactrelations} verify the remaining assumptions of
Lemma~\ref{lem:completion}. Fix one sufficiently small
$\varepsilon>0$. The lemma changes $\omega_\varepsilon$ only for
large $t$ and then allows us to choose $\beta<\beta_*$. Applying
Proposition~\ref{prop:criterion}, we obtain Theorem~\ref{thm:main}.
This completes its proof.

\section{The cases $q\geqslant m$}\label{sec:upper}

\subsection{The case $q=m$}

We first treat the borderline case $q=m$.

\begin{proposition}\label{prop:qequalsm}
Let $1<m<n$ and $p\in\mathbb R$. Every positive
$C^1_{\mathrm{loc}}$ weak solution of
\[
-\Delta_m u=u^p|Du|^m
\qquad\text{in }\mathbb R^n
\]
is constant.
\end{proposition}

\begin{proof}
For $s>0$, define
\[
g(s)=
\begin{cases}
\displaystyle
\exp\left(\frac{s^{p+1}}{p+1}\right),
& p\ne-1,\\[6pt]
s,
& p=-1,
\end{cases}
\]
and
\[
F(s)=\int_0^s g(\tau)^{1/(m-1)}\,d\tau.
\]
Then
\[
g'(s)=s^p g(s),\qquad F'(s)>0,
\qquad F'(s)^{m-1}=g(s).
\]
The integral defining $F$ is finite at the origin for every
$p\in\mathbb R$. Indeed, its integrand tends to $1$ if $p>-1$,
equals $s^{1/(m-1)}$ if $p=-1$, and tends to zero faster than any
power if $p<-1$.

Let $\eta\in C_c^\infty(\mathbb R^n)$. Since $u$ is positive,
$g(u)\eta$ is an admissible test function by standard
approximation. The weak equation yields
\[
\int_{\mathbb R^n}
g(u)|Du|^{m-2}Du\cdot D\eta
+
\int_{\mathbb R^n}
g'(u)|Du|^m\eta
=
\int_{\mathbb R^n}
u^p g(u)|Du|^m\eta.
\]
Since $g'(u)=u^p g(u)$, the last two integrals cancel. Hence
\[
\int_{\mathbb R^n}
g(u)|Du|^{m-2}Du\cdot D\eta=0.
\]
Moreover,
\[
|D(F(u))|^{m-2}D(F(u))
=
g(u)|Du|^{m-2}Du.
\]
Therefore $F(u)$ is a positive weakly $m$-harmonic function on
$\mathbb R^n$.

A positive entire $m$-harmonic function is constant. This follows,
for example, from the scale-invariant Harnack inequality and the
interior oscillation estimate. Thus $F(u)$ is constant. Since
$F'(s)>0$, the function $u$ is constant as well.
\end{proof}

\subsection{The case $q>m$}

Lu and Zhu proved that, for $q>m>1$ and every $p\in\mathbb R$,
each positive entire $C^1$ weak solution of
\[
-\Delta_m u=u^p|Du|^q
\qquad\text{in }\mathbb R^n
\]
is constant; see \cite[Theorem~1.4 and
Corollary~1.5]{LuZhu}.

Combining this result with Theorem~\ref{thm:main} and
Proposition~\ref{prop:qequalsm}, we obtain
Corollary~\ref{cor:complete}.

\noindent\textbf{Acknowledgments:}
The authors would like to thank Professor Marie-Fran\c coise Bidaut-V\'eron for her nice advise and encouragement in this paper. Tian Wu was supported by the National Natural Science Foundation
of China (Grant No. 12601391) and the Fundamental Research Funds for
the Central Universities (Grant No. WK0010250106).
Hua Zhu was supported by the National Natural Science Foundation of
China (Grant No. 12501273, 12661041) and the Research Foundation
of Southwest University of Science and Technology
(Grant No. 25zx7153). Tian Wu and Hua Zhu were also supported by
the Open Research Fund of Hubei Key Laboratory of Mathematical
Sciences (Grant No. MPL2026ORG005).




\noindent\textbf{Use of Large Language Models, AI and Machine Learning Tools:}
We used ChatGPT (OpenAI) as an auxiliary tool in preparing this manuscript. ChatGPT assisted us in carrying out computations and developing arguments within the framework of our ideas and proofs. The authors have carefully checked all AI-generated content and take full responsibility for the mathematical statements, proofs, references, and conclusions presented in this manuscript. We have also made every effort to improve the clarity and readability of the manuscript.

The authors state no conflict of interest.


\bibliographystyle{amsplain}
\bibliography{references}

\footnotesize{
Contact information:
\begin{itemize}
    \item Tian Wu, School of Mathematical Sciences,
    University of Science and Technology of China,
    Hefei, Anhui, 230026, People's Republic of China.
    Email: \emph{wt1997@ustc.edu.cn}

    \item Jin Yan, Institute of Mathematics,
    Academy of Mathematics and Systems Science,
    Chinese Academy of Sciences,
    Beijing, 100190, People's Republic of China.
    Email: \emph{yanjin@amss.ac.cn}

    \item Hua Zhu, School of Mathematical and Physics,
    Southwest University of Science and Technology,
    Mianyang, Sichuan, 621010, People's Republic of China.
    Email: \emph{zhuhmaths@mail.ustc.edu.cn}
\end{itemize}
}

\end{document}